\documentclass[11pt]{article}
\usepackage[a4paper,top=2.54cm,bottom=2.54cm,left=3.17cm,right=3.17cm]{geometry}
\usepackage[shortlabels]{enumitem}
\setenumerate[1]{itemsep=0pt,partopsep=0pt,parsep=\parskip,topsep=0pt}
\usepackage{amsmath,amsthm,amssymb,mathtools}
\usepackage[hidelinks]{hyperref}
\usepackage[numbers,sort&compress]{natbib}
\usepackage{indentfirst}
\allowdisplaybreaks

\newtheorem{theorem}{Theorem}[section]
\newtheorem{lemma}[theorem]{Lemma}

\newtheorem{corollary}[theorem]{Corollary}

\newtheorem{conjecture}{Conjecture}[section]

\newtheorem{proposition}[theorem]{Proposition}

\usepackage{authblk}

\newcommand{\gpk}{\gamma_{P,k}}
\newcommand{\card}[1]{\lvert #1\rvert}
\newcommand{\components}{\mathcal{C}}
\newcommand{\coverfamily}{\mathcal{R}}
\newcommand{\stronggraph}{G_{\mathrm{s}}}

\title{Tight bounds for generalized power domination in regular graphs}
\author[a]{Hangdi Chen\thanks{E-mail: \texttt{chenhangdi188@126.com}}}
\author[b]{Changhong Lu\thanks{E-mail: \texttt{chlu@math.ecnu.edu.cn}}}
\author[b]{Qingjie Ye\thanks{E-mail: \texttt{qjye@math.ecnu.edu.cn}}}
\affil[a]{Fujian Key Laboratory of Financial Information Processing, Key Laboratory of Applied Mathematics of Fujian Province University, Putian University, Fujian Putian 351100, China}
\affil[b]{School of Mathematical Sciences, Key Laboratory of MEA (Ministry of Education), Shanghai Key Laboratory of PMMP, Nantong Institute for Applied Mathematics and Artificial Intelligence, East China Normal University, Shanghai 200241, China}
\date{}

\begin{document}
\maketitle

\begin{abstract}
Dorbec et al.\ [\emph{SIAM J. Discrete Math.}, 27 (2013)] conjectured that, for all integers $k\geq1$ and $r\geq3$, every connected $r$-regular graph $G$ of order $n$, other than $K_{r,r}$, satisfies $\gamma_{P,k}(G)\leq n/(r+1)$.  After disproving this conjecture, Chen et al.\ [\emph{Graphs Combin.}, 38 (2022)] proposed a corresponding conjecture for claw-free regular graphs.
In this paper, we prove this conjecture: for integers $k\geq\ell\geq1$, every connected claw-free $(k+\ell+1)$-regular graph $G$ of order $n$ satisfies $\gamma_{P,k}(G)\leq n/(k+\ell+2)$, and this bound is tight.  Moreover, without the claw-free assumption, we show that, for each fixed integer $k\geq1$, the supremum of $\gamma_{P,k}(G)/\lvert V(G)\rvert$ over all connected $r$-regular graphs $G$ is asymptotic to $(\ln r)/r$ as $r\to\infty$.
\end{abstract}



\section{Introduction}

Power domination arose from the problem of placing as few phasor measurement units as possible while monitoring an electrical network. Baldwin et al.\ \cite{BaldwinMiliBoisenAdapa1993} formulated the monitoring problem, and Haynes et al.\ \cite{HaynesEtAl2002} recast it as the graph-theoretic parameter now called power domination. See also the survey of Dorbec \cite{Dorbec2020}.

Power domination consists of an initial domination step and a subsequent propagation process, where a monitored vertex with exactly one unmonitored neighbor can monitor that neighbor. This propagation rule distinguishes power domination from domination and enables a small initial set to monitor vertices beyond its closed neighborhood \cite{HaynesEtAl2002,Dorbec2020}. Chang et al.~\cite{ChangEtAl2012} introduced $k$-power domination by allowing a monitored vertex with at most $k$ unmonitored neighbors to monitor all of them. The cases $k=0$ and $k=1$ correspond to domination and power domination, respectively \cite{ChangEtAl2012,Dorbec2020}.

We use the following notation and recall the monitoring process. All graphs considered here are finite, simple, and undirected. For a graph $G$, write $V(G)$ and $E(G)$ for its vertex and edge sets. The size of $V(G)$ is the \emph{order} of $G$. For every $v\in V(G)$, let $N_G(v)$ and $N_G[v]=N_G(v)\cup\{v\}$ denote its \emph{open} and \emph{closed neighborhoods}. For $S\subseteq V(G)$, let $N_G[S]=\bigcup_{v\in S}N_G[v]$. A set $S$ is a \emph{dominating set} if $N_G[S]=V(G)$, and the minimum cardinality of such a set is the \emph{domination number} $\gamma(G)$. A graph is \emph{claw-free} if it has no induced copy of $K_{1,3}$. We omit the subscript $G$ when the graph is clear.

Let $k\geq0$, $i\ge 0$ and $S\subseteq V(G)$.  The set of vertices monitored by $S$ at step $i$ is defined recursively by the following rules:

(1) $P_G^{0}(S)= N_G[S]$;

(2) $P_G^{i+1}(S)=\bigcup \{N_G[v]: v\in P_G^{i}(S)$ such that $|N_G[v] \setminus P_G^{i}(S)|\leq k\}$.

\noindent If $P_G^{i_0}(S)=P_G^{i_0+1}(S)$ for some $i_0$, then $P_G^{j}(S)=P_G^{i_0}(S)$ for every $j\ge i_0$ and we accordingly define $P_G^{\infty}(S)=P_G^{i_0}(S)$. If $P_G^\infty(S)=V(G)$, then $S$ is a \emph{$k$-power dominating set} of $G$, abbreviated \emph{$k$-PD-set}.  The minimum cardinality of a $k$-PD-set in $G$ is the \emph{$k$-power domination number} $\gamma_{P,k}(G)$.

For a connected $(k+1)$-regular graph $G$, it is clear that $\gamma_{P,k}(G)=1$. Then it is natural to study the $k$-power domination number of $(k+2)$-regular graphs. Zhao et al.~\cite{ZhaoKangChang2006} proved that if $G$ is a $3$-regular claw-free graph on $n$ vertices, then $\gamma_{P,1}(G)\le n/4$. Chang et al.~\cite{ChangEtAl2012} generalized this result to $(k+2)$-regular claw-free graphs. Dorbec et al.~\cite{DorbecEtAl2013} removed the claw-free condition and showed that every connected $(k+2)$-regular graph $G\neq K_{k+2,k+2}$ of order $n$ satisfies $\gamma_{P,k}(G)\leq n/(k+3)$, and the bound is sharp. They then posed the following conjecture.

\begin{conjecture}[Dorbec et al.\ \cite{DorbecEtAl2013}]\label{DorbecConjecture}
Let $k\geq1$ and $r\geq3$.  If $G\neq K_{r,r}$ is a connected $r$-regular graph of order $n$, then $\gamma_{P,k}(G)\leq n/(r+1)$.
\end{conjecture}

The result of Dorbec et al.~\cite{DorbecEtAl2013} implies that Conjecture~\ref{DorbecConjecture} holds when $k\geq r-2$. Lu et al.~\cite{LuMaoWang2020} showed that it fails for $k=1$ and every even $r\geq4$. Chen et al.~\cite{ChenLuYe2022} proved that it fails throughout the remaining range: for every $r\geq4$ and $1\leq k\leq r-3$, they constructed a connected $r$-regular counterexample. Yang and Wu~\cite{YangWu2022} also obtained counterexamples to Conjecture~\ref{DorbecConjecture}.

Chen et al.~\cite{ChenLuYe2022} further constructed claw-free $r$-regular counterexamples for $k<\lfloor r/2\rfloor$. These examples show that the restriction to claw-free graphs can yield the conjectured bound only when $k\geq\lfloor r/2\rfloor$. Motivated by this threshold, Chen et al.~\cite{ChenLuYe2022} proposed the following conjecture for claw-free regular graphs. Writing $r=k+\ell+1$, the condition $k\geq\lfloor r/2\rfloor$ is equivalent to $k\geq\ell$.
\begin{conjecture}[Chen et al.\ \cite{ChenLuYe2022}]\label{ChenLuYeConj}
For integers $k\ge \ell\ge 1$, if $G$ is a connected claw-free $(k+\ell+1)$-regular graph of order $n$, then $\gpk(G)\le n/(k+\ell+2)$ and the bound is tight.
\end{conjecture}
The case $\ell=1$ follows from an earlier result of Chang et al.\ \cite{ChangEtAl2012}, and Chen et al.~\cite{ChenLuYe2022} proved the cases $\ell\in\{2,3\}$. In this paper, we confirm Conjecture~\ref{ChenLuYeConj} completely.

\begin{theorem}\label{thm:main}
Let $k\geq\ell\geq1$ be integers.  If $G$ is a connected claw-free $(k+\ell+1)$-regular graph of order $n$, then $\gpk(G)\le n/(k+\ell+2)$ and the bound is tight.
\end{theorem}

Without the claw-free assumption, determining the best possible upper bound for $\gamma_{P,k}(G)$ in connected $r$-regular graphs remains open for $1\le k\le r-3$. Dorbec \cite{Dorbec2020} posed this problem for connected regular graphs. To investigate its asymptotic behavior, we define the $k$-power domination ratio as
\[
c_k(r)=\sup\left\{\frac{\gamma_{P,k}(G)}{|V(G)|}:G\text{ is a connected
$r$-regular graph}\right\}.
\]
In this paper, we also determine the asymptotic behavior of $c_k(r)$.
\begin{theorem}\label{thm:asymptotic}
Fix an integer $k\geq1$, and let $q=2\lceil(k+1)/2\rceil$ and $d=\lfloor r/q\rfloor$. For every $\varepsilon>0$ and all sufficiently large integers $r$,
\begin{equation}\label{eq:c-bounds}
  (1-\varepsilon)\frac{\ln d}{qd}\leq c_k(r)\leq 1-\frac{r}{(r+1)^{1+1/r}}.
\end{equation}
Moreover, $c_k(r)\sim(\ln r)/r$ as $r\to\infty$.
\end{theorem}

For each fixed $k$, Conjecture~\ref{DorbecConjecture} would imply $c_k(r)=1/r$ for all sufficiently large $r$, with equality attained by $K_{r,r}$. In contrast, Theorem~\ref{thm:asymptotic} shows that the $k$-power domination problem has the same asymptotic behavior as the domination problem, namely $(\ln r)/r$.

The remainder of the paper is organized as follows. In Section~\ref{sec:proof}, we prove Theorem~\ref{thm:main} using the structure of strong components in claw-free regular graphs. In Section~\ref{sec:unrestricted}, we remove the claw-free assumption and prove Theorem~\ref{thm:asymptotic}.

\section{A sharp bound for claw-free regular graphs}
\label{sec:proof}

In this section, fix integers $k\geq\ell\geq1$ and a connected claw-free $r$-regular graph $G$, where $r=k+\ell+1$. An edge $uv\in E(G)$ is \emph{strong} when $\card{N_G(u)\cap N_G(v)}\geq\ell$ and \emph{weak} otherwise. Let $\stronggraph$ be the spanning subgraph of $G$ whose edge set consists of the strong edges. A connected component of $\stronggraph$ is called a \emph{strong component} of $G$. For each $v\in V(G)$, let $S(v)$ and $W(v)$ denote the \emph{strong} and \emph{weak neighborhoods} of $v$, respectively. More precisely, $S(v)=\{u\in N_G(v):uv\text{ is strong}\}$ and $W(v)=\{u\in N_G(v):uv\text{ is weak}\}$.

The following lemma gives the local structure behind the strong-component decomposition: the weak neighbors of a vertex form a small clique.

\begin{lemma}\label{lem:weak-clique}
For every $v\in V(G)$, $W(v)$ is a clique of $G$. Moreover, $|W(v)|\le \ell$ and $|S(v)|\ge k+1$.
\end{lemma}

\begin{proof}
Let $H=\overline{G[N_G(v)]}$. If $H$ contains a triangle with vertices $a,b,c$, then $v$ together with $a,b,c$ induces a claw in $G$ with center $v$, contradicting that $G$ is claw-free. So $H$ is triangle-free. For every $u\in W(v)$, the definition of $W(v)$ gives $\card{N_G(u)\cap N_G(v)}\leq\ell-1$. As $H$ has $r$ vertices,
\begin{equation}\label{eq:complement-degree}
  d_H(u)
  =(r-1)-\card{N_G(u)\cap N_G(v)}
  \geq (r-1)-(\ell-1)=r-\ell=k+1.
\end{equation}
If $|W(v)|\le 1$, then $W(v)$ is a clique. Now suppose that $|W(v)|\ge 2$. If $W(v)$ were not a clique, there would be nonadjacent vertices $u,w\in W(v)$. Then $uw\in E(H)$. Since $H$ is triangle-free, we have $N_H(u)\cap N_H(w)=\varnothing$. It follows that $d_H(u)+d_H(w)\leq\card{V(H)}=r$. On the other hand, \eqref{eq:complement-degree} gives $d_H(u)+d_H(w)\geq2k+2$, and so $r\ge 2k+2$. However, $k\geq\ell$ yields $r=k+\ell+1\leq2k+1$, a contradiction. Thus $W(v)$ is a clique.

If $|W(v)|\le 1$, then $|W(v)|\le \ell$. If $|W(v)|\ge 2$, choose $x\in W(v)$.  The vertices of $W(v)\setminus\{x\}$ are common neighbors of $x$ and $v$, so
\[
  |W(v)|-1=|W(v)\setminus\{x\}|\leq\card{N_G(x)\cap N_G(v)}\leq\ell-1.
\]
Therefore $|W(v)|\leq\ell$.  Since $d_G(v)=r$, we have \[|S(v)|=d_G(v)-|W(v)|\ge r-\ell=k+1.\qedhere\]
\end{proof}

As an immediate consequence, the neighbors of a vertex that lie outside its strong component also form a clique.

\begin{corollary}\label{cor:outside-clique}
If $H$ is a strong component of $G$ and $v\in V(H)$, then $N(v)\setminus V(H)$ is a clique.
\end{corollary}

\begin{proof}
Since $H$ is a strong component of $G$, every edge from $v$ to a vertex outside $H$ is weak. Hence $N(v)\setminus V(H)\subseteq W(v)$, and the conclusion follows from Lemma~\ref{lem:weak-clique}.
\end{proof}

Let $\components$ be the set of all strong components.  For a vertex $v\in V(G)$, let $C(v)$ denote the unique strong component containing $v$ and define \[F(v)=\{C(v)\}\cup\{H\in\components:N_G(v)\cap V(H)\neq\varnothing\}.\]
A subfamily $\coverfamily\subseteq\components$ is a \emph{strong-component cover} if $F(v)\cap\coverfamily\neq\varnothing$ for every $v\in V(G)$.
Equivalently, $V(G)=\bigcup_{H\in\coverfamily}N_G[V(H)]$.

Since the finite family $\components$ is itself a strong-component cover, choose an inclusion-minimal subcover $\coverfamily\subseteq\components$. To turn $\coverfamily$ into a small $k$-PD-set, we need suitable representatives, disjoint allocations of vertices, and a propagation property. These ingredients are established in Lemmas~\ref{lem:internal-private}--\ref{lem:activate-component}.

\begin{lemma}\label{lem:internal-private}
For every $C\in\coverfamily$, there exists $v_C\in V(C)$ such that $F(v_C)\cap\coverfamily=\{C\}$. Moreover, $v_C$ has no neighbor in any member of $\coverfamily\setminus\{C\}$.
\end{lemma}

\begin{proof}
Since $\coverfamily$ is inclusion-minimal, $\coverfamily\setminus\{C\}$ is not a strong-component cover.  Hence there is a vertex $p$ such that $F(p)\cap(\coverfamily\setminus\{C\})=\varnothing$. Since $\coverfamily$ covers $p$, it follows that $F(p)\cap\coverfamily=\{C\}$.
If $p\in V(C)$, set $v_C=p$.

Now suppose that $p\notin V(C)$.  Since $C\in F(p)$, choose $v\in N_G(p)\cap V(C)$.  If $v$ had a neighbor $x$ in some $C'\in\coverfamily\setminus\{C\}$, then both $p$ and $x$ would lie in $N_G(v)\setminus V(C)$.  By Corollary~\ref{cor:outside-clique}, $px\in E(G)$.  This would put $C'$ in $F(p)\cap\coverfamily$, contradicting $F(p)\cap\coverfamily=\{C\}$. Thus $F(v)\cap\coverfamily=\{C\}$, and we take $v_C=v$.
\end{proof}

For each $C\in\coverfamily$, fix a representative $v_C$ supplied by Lemma~\ref{lem:internal-private}. The following lemma shows that the private representatives have pairwise disjoint closed neighborhoods.

\begin{lemma}\label{lem:allocation}
If $C,C'\in\coverfamily$ and $C\neq C'$, then $N_G[v_C]\cap N_G[v_{C'}]=\varnothing$. Consequently, $n\geq(r+1)\card{\coverfamily}$.
\end{lemma}

\begin{proof}
Let $C,C'\in\coverfamily$ be distinct. Then the vertices $v_C$ and $v_{C'}$ are distinct. They are not adjacent by Lemma~\ref{lem:internal-private}.

Suppose that $x\in N_G(v_C)\cap N_G(v_{C'})$, and let $D=C(x)$. If $D\in\coverfamily$, then $D\in F(v_C)\cap\coverfamily$ and $D\in F(v_{C'})\cap\coverfamily$. Lemma~\ref{lem:internal-private} would give $D=C=C'$, a contradiction.

Now suppose that $D\notin\coverfamily$. Since $C,C'\in\coverfamily$, both $v_C$ and $v_{C'}$ lie in $N_G(x)\setminus V(D)$. This set is a clique by Corollary~\ref{cor:outside-clique}, so $v_Cv_{C'}\in E(G)$, again contradicting Lemma~\ref{lem:internal-private}. Thus the closed neighborhoods are pairwise disjoint. Since $G$ is $r$-regular, each of them has order $r+1$, and the final inequality follows.
\end{proof}

The following lemma shows that one initially selected vertex monitors the closed neighborhood of its entire strong component.

\begin{lemma}\label{lem:activate-component}
Let $H\in \coverfamily$. For each $y\in V(H)$, $N_G[V(H)]\subseteq P_G^\infty(\{y\})$.
\end{lemma}

\begin{proof}
Note that $P_G^0(\{y\})=N_G[y]$.  Consider a strong edge $uv$ and suppose that $N_G[u]$ has already been monitored. Then $u$ and every vertex in $N_G(u)\cap N_G(v)$ are monitored. Since $uv$ is strong, the vertex $v$ therefore has at least $\ell+1$ monitored neighbors. It follows that the number of its unmonitored neighbors is at most
\[
  r-(\ell+1)=(k+\ell+1)-(\ell+1)=k.
\]
Thus $N_G[v]$ will be monitored by $P_G^\infty(\{y\})$.

For any $x\in V(H)$, choose a path $y=x_0,x_1,\ldots,x_t=x$ in $\stronggraph[V(H)]$.  Starting with $N_G[x_0]$ monitored and applying the preceding observation successively to the strong edges $x_{i-1}x_i$, we see that $N_G[x_i]$ is eventually monitored for every $i$.  Taking the union over $x\in V(H)$ gives $N_G[V(H)]\subseteq P_G^\infty(\{y\})$.
\end{proof}

We are now ready to prove Theorem~\ref{thm:main}.

\begin{proof}[Proof of Theorem~\ref{thm:main}]

Let $S=\{u_H:H\in\coverfamily\}$. By Lemma~\ref{lem:activate-component}, $N_G[V(H)]\subseteq P_G^\infty(\{u_H\})$ for every $H\in\coverfamily$. It follows that $P_G^\infty(\{u_H\})\subseteq P_G^\infty(S)$.  The cover condition now yields
\[
  V(G)=\bigcup_{H\in\coverfamily}N_G[V(H)]\subseteq P_G^\infty(S),
\]
so $S$ is a $k$-PD-set. By Lemma~\ref{lem:allocation}, $n\ge (r+1)\card{\coverfamily}$.
Since $\card{S}=\card{\coverfamily}$ and $r=k+\ell+1$,
\[
  \gpk(G)\leq\card{S}=\card{\coverfamily}
  \leq\frac{n}{r+1}
  =\frac{n}{k+\ell+2}.
\]
Moreover, Chen et al.~\cite{ChenLuYe2022} showed that the bound is tight.\qedhere
\end{proof}

\section{Asymptotically sharp bounds for regular graphs}
\label{sec:unrestricted}

In this section, we remove the claw-free assumption and study the $k$-power domination ratio for a fixed integer $k\geq1$. Before proving Theorem~\ref{thm:asymptotic}, we recall two results on domination. The first supplies regular graphs with large domination ratio for the lower-bound construction, and the second gives the universal upper bound through domination.

\begin{theorem}[Alon and Wormald \cite{AlonWormald2010}]\label{thm:large-domination}
For every $\varepsilon>0$ and all sufficiently large integers $d$, there exists a finite simple $d$-regular graph $H$ such that
\[
  \gamma(H)\geq(1-\varepsilon)\frac{\ln d}{d}\card{V(H)}.
\]
\end{theorem}

\begin{theorem}[Caro and Roditty \cite{CaroRoditty1985}]\label{thm:caro-roditty}
Let $G$ be a graph of order $n$ and minimum degree $\delta\geq2$.  Then
\[
  \gamma(G)\leq\left(1-\frac{\delta}{(\delta+1)^{1+1/\delta}}\right)n.
\]
\end{theorem}

The third ingredient is a propagation obstruction. A nonempty set $F\subseteq V(G)$ is a \emph{$k$-fort} if every vertex in $N_G[F]\setminus F$ has at least $k+1$ neighbors in $F$. The relevance of this definition is captured by the following result.

\begin{proposition}[Chen et al.\ \cite{ChenLuYe2022}]\label{prop:fort-obstruction}

If $F$ is a $k$-fort of a graph $G$ and $S$ is a $k$-PD-set of $G$, then
$S\cap N_G[F]\neq\varnothing$.
\end{proposition}

We are now ready to prove Theorem~\ref{thm:asymptotic}.

\begin{proof}[Proof of Theorem~\ref{thm:asymptotic}]
Fix $\varepsilon>0$.
Every dominating set is a $k$-PD-set. Hence, for every connected $r$-regular graph $G$, Theorem~\ref{thm:caro-roditty} applied with minimum degree $r$ gives
\[
  \gamma_{P,k}(G)\leq\gamma(G)\leq\left(1-\frac{r}{(r+1)^{1+1/r}}\right)\card{V(G)}.
\]
So the upper bound in \eqref{eq:c-bounds} holds.

For the lower bound, recall that $q=2\lceil(k+1)/2\rceil$ and $d=\lfloor r/q\rfloor$, and put $s=r-qd$.  Thus $q$ is even, $q\geq k+1$, and $0\leq s\leq q-1$.
Let $J$ be a simple $s$-regular graph on $q$ vertices. Such a graph exists since $q$ is even and $0\leq s\leq q-1$. For instance, arrange the $q$ vertices on a circle and join each vertex to its $\lfloor s/2\rfloor$ nearest vertices in each direction. When $s$ is odd, also join each pair of antipodal vertices. Take a graph $H$ supplied by Theorem~\ref{thm:large-domination}. Since the domination number is additive over connected components, some connected component of $H$ satisfies the same lower bound. Replacing $H$ by this component, we may assume that $H$ is connected. Replace each vertex $v$ of $H$ by a copy $F_v$ of $J$, and join every vertex of $F_u$ to every vertex of $F_v$ whenever $uv\in E(H)$. The resulting graph $G$ is connected because $H$ is connected. Each vertex of $F_v$ has $s$ neighbors within $F_v$ and $qd$ neighbors in the copies corresponding to the neighbors of $v$ in $H$, for a total of $s+qd=r$ neighbors. Thus $G$ has order $q\card{V(H)}$ and is $r$-regular.

We next show that each $F_v$ is a $k$-fort. By construction, if $x\in N_G[F_v]\setminus F_v$, then $x\in F_u$ for some $u\in N_H(v)$. All possible edges between $F_u$ and $F_v$ were added, so $x$ is adjacent to every vertex of $F_v$. Consequently, $\card{N_G(x)\cap F_v}=q\geq k+1$. Hence $F_v$ is a $k$-fort.

Let $S$ be any $k$-PD-set of $G$. By Proposition~\ref{prop:fort-obstruction}, $S\cap N_G[F_v]\neq\varnothing$ for every $v\in V(H)$. Let $T=\{v\in V(H):S\cap F_v\neq\varnothing\}$. For each $v\in V(H)$, the condition $S\cap N_G[F_v]\neq\varnothing$ implies that $S\cap F_u\neq\varnothing$ for some $u\in N_H[v]$. Thus $T\cap N_H[v]\neq\varnothing$ for every $v\in V(H)$, so $T$ dominates $H$. Since the copies $F_u$ are pairwise disjoint, $\card{T}\leq\card{S}$. By Theorem~\ref{thm:large-domination}, we have
\[
  \card{S}\geq\card{T}\geq\gamma(H)\geq(1-\varepsilon)\frac{\ln d}{d}\card{V(H)},
\]
and division by $\card{V(G)}=q\card{V(H)}$ proves the lower bound in \eqref{eq:c-bounds}.  Since $q$ is fixed as $r\to\infty$, the two bounds have the same leading term:
\[
  1-\frac{r}{(r+1)^{1+1/r}}=1-\frac{r}{r+1}e^{-\frac{\ln (r+1)}{r}}=1-e^{-\frac{\ln r}{r}}+O\left(\frac{1}{r}\right)=(1+o(1))\frac{\ln r}{r}
\]
and
\[\frac{\ln\lfloor r/q\rfloor}{q\lfloor r/q\rfloor}=(1+o(1))\frac{\ln r}{r}.\]
Since $\varepsilon>0$ is arbitrary, the asymptotic conclusion follows.
\end{proof}

\section*{Acknowledgments}

This work was supported by the National Natural Science Foundation of China (No. 12331014), Science and Technology Commission of Shanghai Municipality (No. 22DZ2229014), Natural Science Foundation of Fujian Province, China (No. 2024J01875), Science-Technology Foundation of Putian University (No. 2023059), and the Fundamental Research Funds for the Central Universities.

\end{document}